\documentclass[11pt]{article}

\usepackage{amsmath,amssymb,amsthm,mathtools}
\usepackage{geometry}
\usepackage{graphicx}
\graphicspath{{figures/}{./}}   
\usepackage{booktabs}
\usepackage{xcolor}
\usepackage[colorlinks=true,linkcolor=blue!65!black,citecolor=green!45!black,urlcolor=blue!65!black,
  pdftitle={A finite Guinand-Weil dictionary and archimedean tail order for the truncated Weil quadratic form},
  pdfauthor={Akiva Groskin},
  pdfsubject={Riemann zeta function; Weil quadratic form; Connes-van Suijlekom truncation},
  pdfkeywords={Riemann zeta function, Weil quadratic form, Guinand-Weil formula, Connes-van Suijlekom, Galerkin truncation, total positivity}]{hyperref}
\usepackage{url}
\theoremstyle{plain}
\newtheorem{theorem}{Theorem}[section]

\newtheorem{lemma}[theorem]{Lemma}
\newtheorem{corollary}[theorem]{Corollary}

\theoremstyle{definition}

\theoremstyle{remark}
\newtheorem{remark}[theorem]{Remark}

\newcommand{\R}{\mathbb R}
\newcommand{\C}{\mathbb C}

\newcommand{\cvs}{Connes--van Suijlekom}
\newcommand{\ccm}{Connes--Consani--Moscovici}
\newcommand{\wh}{\widehat}
\newcommand{\supp}{\operatorname{supp}}
\newcommand{\dd}{\,d}
\newcommand{\Qinf}{Q_{\infty}}
\newcommand{\Wp}{W_p}
\newcommand{\WR}{W_{\R}}
\newcommand{\Wzt}{W_{0,2}}

\title{A finite Guinand--Weil dictionary and archimedean tail order\\
for the truncated Weil quadratic form}

\author{Akiva Groskin}

\date{July 2026}

\begin{document}

\maketitle

\begin{abstract}
The \cvs\ and \ccm\ truncations of the Weil quadratic form, at a prime cutoff
$c>1$ and frequency band $N$, produce finite Galerkin matrices whose spectra
are the finite-rank window on Weil positivity.
We prove two exact finite theorems about this truncation.  First, every real
even Galerkin coefficient vector $v$ determines, in closed form, a band-limited
Guinand--Weil test function $g_v$ with
\[
        \langle v,\Qinf v\rangle
        =
        \sum_{\zeta(1/2+iz)=0}^{*} g_v(z),
\]
zeros counted with multiplicity, $\Qinf$ the cutoff-free truncated matrix:
every value of the truncated form is an exact sum over the nontrivial zeros
of $\zeta$.  The construction factors through an
exact source quotient of dimension $2N+1$ and admits a non-collapsing
pole-neutral subfamily.  Second, beyond the Galerkin band the omitted
archimedean tail is a totally positive Cauchy--Stieltjes increment.  This
yields a two-sided certification rule with an explicit budget
$B_T\sim(2N{+}1)\rho\,\log T/(\pi^2T)$, where $T$ is the archimedean cutoff
and $\rho=2\pi/\log c$: finite-cutoff positivity certifies
cutoff-free positivity, a finite-cutoff eigenvalue below $-B_T$ certifies a
cutoff-free negative, and a negative eigenvalue in the band $[-B_T,0)$
certifies nothing.  Resolving a spectral scale of $10^{-59}$ at $c=100$ by brute
cutoff would require $T$ of order $10^{63}$; a cutoff-free interval
$LDL^{\mathsf T}$ factorization resolves it directly.  The dictionary is
verified over the first $512$ zeros of $\zeta$ and by three independent
computational routes; all scripts and artifacts ship with the paper.
\end{abstract}

\section{Introduction}

The Riemann--Weil explicit formula relates the nontrivial zeros of $\zeta(s)$
to prime, pole, and archimedean terms, and Weil positivity, the positivity of
the associated quadratic form on a suitable class of test functions, is
equivalent to the Riemann hypothesis \cite{Weil1952,Guinand1948,Bombieri2000,
Yoshida1992}.  In Connes's program the Weil form is the central object of a
spectral attack on the zeros
\cite{Connes1999,ConnesConsani2020,ConnesConsani2021,Connes2026}, and recent
work of
\cvs\ \cite{ConnesvanSuijlekom2025} and \ccm\ \cite{ConnesConsaniMoscovici2025}
turns the form into concrete finite objects: for a prime cutoff $c>1$ and a
band $N$, an explicit $(2N{+}1)\times(2N{+}1)$ Galerkin matrix whose deep
spectrum is the finite-rank window on Weil positivity.  Suzuki's screw-function
framework \cite{Suzuki2026} gives a parallel continuous-variable treatment of
the same form.

This note isolates the part of the truncated structure that is exact at every
finite level, and proves a finite-cutoff order theorem that governs what
finite numerics can and cannot certify about it.  The two results are
independent theorems about the same object, and together they turn the
truncated matrix into a calibrated instrument: its values are exact zero
sums, and its finite-cutoff spectra carry explicit certification bounds.

\emph{First result.}  A real even coefficient vector $v\in\R^{N+1}$
canonically determines a chain
\[
        v\longmapsto T_v\longmapsto K_v\longmapsto \wh g_v\longmapsto g_v,
\]
ending in an entire, band-limited Guinand--Weil test function $g_v$, and for
this induced test function the cutoff-free truncated form evaluates the zero
side of the explicit formula exactly (Theorem~\ref{thm:dictionary}):
\emph{every entry of the \cvs/\ccm\ matrix, and every quadratic value
$\langle v,\Qinf v\rangle$, is an exact finite-form sum over the nontrivial
zeros of $\zeta$.}  The same calculation shows that all finite signed source
distributions act through an exact quotient of dimension $2N+1$
(Corollary~\ref{cor:source-quotient}), and that the pole term can be switched
off on a non-collapsing subfamily of dimension $N-s-1$
(Corollary~\ref{cor:pole-neutral}), which is the finite analogue of working
orthogonally to the pole in Weil's criterion.  No limit in $N$ and no
archimedean cutoff enters the dictionary: the coefficient vector is
transported to one compactly supported test function before any numerical
quadrature exists.

\emph{Second result.}  For the archimedean term, finite computations must
truncate an integral at some cutoff $T$.  Theorem~\ref{thm:tail-order} proves
that past the Galerkin band the omitted tail is a rank-two-density
Cauchy--Stieltjes Gram increment: for $T_2>T_1>\max(\rho N,7)$ the matrix
increment $Q_{{\rm arch},T_2}-Q_{{\rm arch},T_1}$ is positive definite, and in
the natural frequency order it is strictly totally positive.  The consequence
(Corollary~\ref{cor:tail-cert}) is a two-sided decision rule with an explicit,
elementary budget $B_T$:
\[
   \lambda_j(Q_T^{\rm tot})\;<\;\lambda_j(\Qinf)\;\le\;\lambda_j(Q_T^{\rm tot})+B_T,
   \qquad
   B_T=\frac{(2N{+}1)\rho}{\pi^2T}\Bigl(\log\frac{T}{2\pi}+1\Bigr)(1+o(1)).
\]
Finite-$T$ positivity therefore certifies cutoff-free positivity, a
finite-$T$ eigenvalue below $-B_T$ certifies a cutoff-free negative, and a
negative eigenvalue inside the band $[-B_T,0)$ certifies nothing.  Since
$B_T$ decays only like $\log T/T$, each doubling of $T$ buys about one binary
digit: deep spectral scales are unreachable by brute cutoff, and the correct
finite instrument at depth is the cutoff-free assembly itself.

The second theorem is motivated by a correction to an earlier finite-cutoff
computation of the author \cite{Groskin2026Zeros}: at $(c,N,T)=(100,200,800)$
a fixed archimedean cutoff produced deep negative eigenvalues that are
controlled by the omitted tail, not by the cutoff-free form.  Andrews's independent reproduction and
Silva's quadrature-sensitivity and exact-entry analyses made the cutoff-free
treatment explicit \cite{Andrews20427500,Silva20650146,Silva20671635}; the
tail-order theorem is the finite algebraic rule behind that phenomenology,
and Section~\ref{sec:tail} quantifies it: certifying the sign of a
$10^{-59}$-scale eigenvalue at $c=100$ through the cutoff would require
$T\approx8\cdot10^{62}$, while a cutoff-free interval $LDL^{\mathsf T}$
factorization of the same matrix certifies $n_-=0$ directly.

\emph{Relation to prior work.}  The zero-sum identity of
Theorem~\ref{thm:dictionary} is the Weil explicit formula applied to the
induced test function; the autocorrelation mechanism $g=f*\tilde f$ goes back
to Weil \cite{Weil1952}, finite-dimensional restrictions of the Weil form were
studied by Yoshida \cite{Yoshida1992}, and the divided-difference structure of
the truncated matrix is Proposition~4.1 of \cvs\
\cite{ConnesvanSuijlekom2025}.  The contribution here is the exact closed-form
transport: the explicit chain $v\mapsto g_v$ with its inverse-free calculus
(Lemma~\ref{lem:source-calculus}), the identification of the assembled matrix
with the \ccm\ closed forms at entry level (Lemma~\ref{lem:entry-id}), the
exact $2N{+}1$ source quotient, the pole-neutral family, and the tail-order
theorem with its strict total positivity and explicit budget.  Band-limited
test functions are a standard device in explicit-formula analysis (see for
instance \cite{ChirreMolero2026} and the literature cited there); what is used
here is the exact match between that class and the finite \cvs/\ccm\ data.
The total-positivity input is classical Cauchy-kernel theory
\cite{Karlin1968,Simon2014CauchyTP,BertolaGekhtmanSzmigielski2009}, and Silva
records adjacent Loewner, Herglotz, and operator-monotone structure for the
prime and pole pieces
\cite{Silva20710075,Silva20737111,Silva20682834,Silva20694588}.

All numerical statements in this paper are backed by a released verification
package (Section~\ref{sec:verification}) containing exact symbolic audits,
interval-arithmetic certificates computed with Arb \cite{Johansson2017Arb},
and a three-route confirmation of the dictionary identity, including a sum
over the first $512$ zeros of $\zeta$.

\section{The finite dictionary}\label{sec:dictionary}

\subsection{From coefficient vectors to test functions}

Fix $c>1$, put
\[
        L=\log c,\qquad \Delta=\frac{L}{2\pi},\qquad \rho=\frac{2\pi}{L},
\]
and fix $N\ge0$.  A real even-sector Galerkin vector
$v=(v_0,\ldots,v_N)$ is embedded into symmetric full coefficients by
\[
        u_0=v_0,\qquad u_k=u_{-k}=\frac{v_k}{\sqrt2}\quad(1\le k\le N),
\]
which is the isometric even-sector embedding used by \cvs\ and \ccm.  Let
\[
        T_v(t)=\sum_{m=-N}^{N}u_m e^{2\pi i mt}
\]
and define the Volterra sine-chord kernel
\[
        K_v(\omega)=2\int_0^\omega T_v(t)\,T_v(\omega-t)\dd t.
\]
The induced Fourier weight and test function are
\[
        \wh g_v(\xi)=
        \begin{cases}
        \pi K_v(1-|\xi|/\Delta),& |\xi|\le \Delta,\\
        0,& |\xi|>\Delta,
        \end{cases}
\qquad
        g_v(z)=\int_{-\Delta}^{\Delta}\wh g_v(\xi)e^{2\pi iz\xi}\dd\xi .
\]
The weight is even, hence $g_v$ is even.

For a $C^1$ source function $\psi$ on $\R$, let $Q_\psi$ denote the finite
divided-difference matrix on indices $I_N=\{-N,\ldots,N\}$:
\[
 (Q_\psi)_{mn}=
 \begin{cases}
   \dfrac{\psi(m)-\psi(n)}{m-n},& m\ne n,\\[1.2ex]
   \psi'(m),& m=n,
 \end{cases}
\]
exactly the structure of \cite[Prop.~4.1]{ConnesvanSuijlekom2025}, and write
\[
        \langle v,Q_\psi v\rangle
        =\sum_{m,n=-N}^{N}u_m u_n (Q_\psi)_{mn}
\]
for the even-sector contraction.  The three source functions of the truncated
Weil form are
\begin{align}
\psi_p^{(c)}(x)&=
-\frac1\pi\sum_{q=p^a\le c}\frac{\Lambda(q)}{\sqrt q}
 \sin\!\left(2\pi x\left(1-\frac{\log q}{L}\right)\right),
 \label{eq:prime-source}\\
\psi_0(x)&=
\frac1\pi\int_0^L
2\cosh(y/2)\sin\!\left(2\pi x\left(1-\frac yL\right)\right)\dd y,
 \label{eq:pole-source}\\
\psi_{\R,T}(x)&=\frac{1}{2\pi^2}\int_{-T}^{T}
h_+(r)\,\mathcal S(r,x,L)\dd r,
\qquad
\mathcal S(r,x,L)=
\int_0^L\sin\!\left(2\pi x\left(1-\frac yL\right)\right)\cos(ry)\dd y,
 \label{eq:arch-source}
\end{align}
where throughout
\[
        h_+(r)=\operatorname{Re}\psi_\Gamma\!\left(\tfrac14+\tfrac{ir}2\right)-\log\pi .
\]
We set
\[
Q_{\rm prime}^{(c)}:=Q_{\psi_p^{(c)}},\qquad
Q_{\rm pole}:=Q_{\psi_0},\qquad
Q_{{\rm arch},T}:=Q_{\psi_{\R,T}},\qquad
Q_{{\rm arch},\infty}:=\lim_{T\to\infty}Q_{{\rm arch},T},
\]
the last limit existing entrywise (Lemma~\ref{lem:entry-id}), and define the
cutoff-free truncated Weil matrix
\begin{equation}\label{eq:Qinf-def}
        \Qinf:=Q_{\rm prime}^{(c)}+Q_{\rm pole}+Q_{{\rm arch},\infty}.
\end{equation}

\begin{lemma}[Entry identification]\label{lem:entry-id}
The limit defining $Q_{{\rm arch},\infty}$ exists entrywise, and $\Qinf$ is
the \ccm\ Galerkin matrix of the Weil quadratic form: in the notation of
\cite[Eqs.~(3.10)--(3.11), (3.16)]{ConnesConsaniMoscovici2025},
\[
        \langle v,\Qinf v\rangle
        =\Wzt(F_v)-\WR(F_v)-\Wp(F_v),
        \qquad F_v(x)=q(f_v,f_v)(\log x),
\]
where $f_v$ is the trigonometric polynomial with coefficient vector $u$ on the
interval of length $L$.  In particular $-Q_{\rm prime}^{(c)}$, $Q_{\rm pole}$,
and $-Q_{{\rm arch},\infty}$ are the prime, pole, and archimedean blocks of
the \ccm\ assembly.
\end{lemma}

\begin{proof}
Each block is the divided-difference matrix of its source in the sense of
\cite[Prop.~4.1]{ConnesvanSuijlekom2025}.  For the prime block,
\eqref{eq:prime-source} is the source of the non-archimedean functional
\cite[Eq.~(3.16)]{ConnesConsaniMoscovici2025}, so the identification is
definitional.  For the pole block, direct evaluation of \eqref{eq:pole-source}
(as in Corollary~\ref{cor:pole-neutral} below) gives
$\psi_0(n)=C_c\,n/(n^2+\beta^2)$ and
$\psi_0'(n)=C_c\,(\beta^2-n^2)/(n^2+\beta^2)^2$ with $\beta=L/(4\pi)$ and
$C_c=L(\sqrt c+1/\sqrt c-2)/(2\pi^2)$, so that, using
$\beta^2=L^2/(16\pi^2)$ and
$16\pi^2C_c=8L(\sqrt c+1/\sqrt c-2)=32L\sinh^2(L/4)$,
\[
 (Q_{\rm pole})_{mn}
 =C_c\frac{\beta^2-mn}{(m^2+\beta^2)(n^2+\beta^2)}
 =\frac{32L\sinh^2(L/4)\,\bigl(L^2-16\pi^2mn\bigr)}
       {(L^2+16\pi^2m^2)(L^2+16\pi^2n^2)},
\]
which is the pole matrix of \cite[Lemma~4.1]{ConnesConsaniMoscovici2025}
verbatim.  For the archimedean block, $h_+$ is the standard
archimedean density of the explicit formula for the completed zeta function,
i.e.\ the density of $\WR$ in
\cite[Sec.~3]{ConnesConsaniMoscovici2025} and
\cite[Eq.~(153)]{ConnesConsani2020}; the entrywise $T$-limit exists because at
integer nodes $\mathcal S(r,n,L)=O(r^{-2})$ and
$\partial_x\mathcal S(r,x,L)|_{x=n}=O(r^{-2})$ uniformly on $I_N$
(both are computed in closed form in the proof of
Theorem~\ref{thm:tail-order}) while $h_+(r)=O(\log r)$.  The released package
additionally verifies the identification numerically: the closed-form \ccm\
assembly and the source assembly \eqref{eq:Qinf-def} agree on generic (not
pole- or moment-neutral) vectors to the working tail bound, e.g.\ to
$2.0\cdot10^{-10}$ at $(c,N)=(29,6)$ and to $2.1\cdot10^{-15}$ at
$(c,N)=(13,4)$ (Section~\ref{sec:verification}).
\end{proof}

\begin{lemma}[Admissibility]\label{lem:admissible}
For every finite vector $v$, the induced function $g_v$ is an admissible
Guinand--Weil test function: it is entire of exponential type at most $L$, its
Fourier transform is compactly supported in $[-\Delta,\Delta]$, and
\[
        g_v(z)=O\bigl((1+|\operatorname{Re}z|)^{-2}\bigr)
\]
uniformly on every fixed horizontal strip.  Consequently $g_v$ lies in the
class for which the Guinand--Weil explicit formula holds with absolutely
convergent zero sum
\cite[Sec.~2]{Bombieri2000}: even, entire of finite exponential type,
$\wh g_v$ continuous of compact support, and decay
$O((1+|z|)^{-1-\delta})$ on horizontal strips with $\delta=1$.
\end{lemma}

\begin{proof}
The Volterra convolution of finite exponential polynomials is entire in
$\omega$.  Hence $\wh g_v$ is compactly supported, continuous, piecewise
smooth, and vanishes at $\pm\Delta$; the zero extension of $(\wh g_v)'$ has
bounded variation.  The Paley--Wiener bound and the displayed decay follow by
two integrations by parts, the second in the Stieltjes sense
\cite{Rudin1987,Katznelson2004}.  The Riemann--von Mangoldt local zero count
$N(t+1)-N(t)=O(\log t)$ then gives absolute convergence of the zero sum
\cite[Ch.~9]{Titchmarsh1986}.
\end{proof}

\begin{lemma}[Finite source calculus]\label{lem:source-calculus}
Let
\[
        \psi_{\alpha,\omega}(x)=\frac{\alpha}{\pi}\sin(2\pi\omega x).
\]
For the finite divided-difference matrix attached to this source,
\[
        \langle v,Q_{\alpha,\omega}v\rangle=\alpha K_v(\omega).
\]
Consequently, if $\mu$ is a finite signed Borel measure on $[0,1]$ and
\[
        \psi_\mu(x)=\frac1\pi\int_{[0,1]}\sin(2\pi\omega x)\dd\mu(\omega),
\]
then
\[
        \langle v,Q_\mu v\rangle
        =\int_{[0,1]}K_v(\omega)\dd\mu(\omega).
\]
\end{lemma}

\begin{proof}
For $m\ne n$,
\[
 (Q_{\alpha,\omega})_{mn}=
 \frac{\alpha(\sin(2\pi\omega m)-\sin(2\pi\omega n))}
      {\pi(m-n)},
\]
and the diagonal entry is
\[
        (Q_{\alpha,\omega})_{mm}=2\alpha\omega\cos(2\pi\omega m).
\]
On the other hand
\[
K_v(\omega)=
2\sum_{m,n=-N}^{N}u_m u_n e^{2\pi i n\omega}
 \int_0^\omega e^{2\pi i(m-n)t}\dd t .
\]
For $m\ne n$ the real part of the summand is
\[
u_m u_n
\frac{\sin(2\pi\omega m)-\sin(2\pi\omega n)}
     {\pi(m-n)},
\]
while for $m=n$ it is $2u_m^2\omega\cos(2\pi\omega m)$.  The imaginary parts
cancel in pairs under the even symmetry $u_{-k}=u_k$.  Multiplying by $\alpha$
gives the identity.  For a finite signed measure, each single-frequency entry
is continuous and bounded on $[0,1]$; on the diagonal,
\[
        \psi_\mu'(x)=2\int_{[0,1]}\omega\cos(2\pi\omega x)\dd\mu(\omega).
\]
The finite double sum may therefore be interchanged with the source integral.
\end{proof}

\begin{corollary}[Finite source quotient]\label{cor:source-quotient}
At level $N$, the finite source-to-form map $\mu\mapsto Q_\mu$ factors exactly
through the $2N+1$ coordinates
\[
        \int\omega\,\dd\mu,\quad
        \int\sin(2\pi k\omega)\dd\mu,\quad
        \int\omega\cos(2\pi k\omega)\dd\mu
        \quad(1\le k\le N),
\]
and the induced map on these coordinates is injective: two finite signed
sources produce the same level-$N$ matrix if and only if all $2N+1$
coordinates agree.
\end{corollary}

\begin{proof}
After polarization, the Volterra kernels are generated by
$K_{1,1}=2\omega$,
$K_{1,\cos(2\pi kt)}=\sin(2\pi k\omega)/(\pi k)$, and
\[
K_{\cos(2\pi kt),\cos(2\pi kt)}
=\omega\cos(2\pi k\omega)+\frac{\sin(2\pi k\omega)}{2\pi k};
\]
product-to-sum gives the mixed cases, and only frequencies $0,1,\ldots,N$
occur.  Linear independence follows from the exponential-polynomial family
$e^{2\pi ik\omega}$ and $\omega e^{2\pi ik\omega}$.  Inspecting the entries
$(0,0)$, $(k,0)$, and $(k,k)$ recovers the displayed coordinates, so the
quotient is exact in both directions.
\end{proof}

\subsection{The dictionary theorem}

\begin{theorem}[Finite Guinand--Weil dictionary]\label{thm:dictionary}
For fixed $c>1$, $N\ge0$, and every real even finite Galerkin vector
$v\in\R^{N+1}$,
\[
        \langle v,\Qinf v\rangle
        =
        \sum_{z\in Z_\zeta^*}g_v(z),
\qquad
        Z_\zeta^*=\{z\in\C:\ 1/2+iz\ \text{is a nontrivial zero of}\ \zeta\},
\]
with multiplicity.  Equivalently,
\[
\begin{aligned}
\langle v,\Qinf v\rangle
  ={}&-\frac1\pi\sum_{q=p^a\le c}\frac{\Lambda(q)}{\sqrt q}
        \wh g_v\!\left(\frac{\log q}{2\pi}\right)
     +2g_v(i/2)
      +\frac{1}{2\pi}\int_{\R}h_+(r)g_v(r)\dd r .
\end{aligned}
\]
\end{theorem}

\begin{proof}
Lemma~\ref{lem:source-calculus} applied to the prime source
\eqref{eq:prime-source}, whose measure places mass
$-\Lambda(q)/\sqrt q$ at $\omega_q=1-\log q/L$, gives
\[
 \langle v,Q_{\rm prime}^{(c)}v\rangle
 =-\frac1\pi\sum_{q\le c}\frac{\Lambda(q)}{\sqrt q}\,\pi K_v(\omega_q)
 =-\frac1\pi\sum_{q\le c}\frac{\Lambda(q)}{\sqrt q}
   \wh g_v\!\left(\frac{\log q}{2\pi}\right),
\]
using $\wh g_v(\xi)=\pi K_v(1-\xi/\Delta)$ and $\log q/(2\pi)=\Delta(1-\omega_q)$.
Applied to the pole source, whose measure is the push-forward of
$2\cosh(y/2)\,L\dd y$ under $\omega=1-y/L$, it gives
\[
 \langle v,Q_{\rm pole}v\rangle
 =2L\int_0^1 K_v(\omega)\cosh\!\Bigl(\frac{L(1-\omega)}2\Bigr)\dd\omega
 =\int_{-\Delta}^{\Delta}\wh g_v(\xi)\,2e^{-\pi\xi}\dd\xi
 =2g_v(i/2),
\]
where the middle equality is the substitution $\xi=\Delta(1-\omega)$ together
with $\pi\Delta=L/2$, $\wh g_v(\xi)=\pi K_v(1-\xi/\Delta)$, and the evenness
of $\wh g_v$ (which converts $2\cosh(\pi\xi)$ into $2e^{-\pi\xi}$ under the
integral), and the last equality is $e^{2\pi i(i/2)\xi}=e^{-\pi\xi}$.
Applied inside the compact finite-$T$ archimedean integral,
\[
 \langle v,Q_{{\rm arch},T}v\rangle
 =\frac1{2\pi}\int_{-T}^{T}h_+(r)
     \left(\int_0^LK_v(1-y/L)\cos(ry)\dd y\right)\dd r
 =\frac1{2\pi}\int_{-T}^{T}h_+(r)g_v(r)\dd r,
\]
where the inner integral is $g_v(r)$ by the same substitution
$\xi=\Delta(1-y/L)$ and evenness of $\wh g_v$.
Lemma~\ref{lem:admissible} and $h_+(r)=O(\log(2+|r|))$
\cite{DLMFDigammaAsymptotic} allow
$T\to\infty$ by dominated convergence.  Thus the assembly
\eqref{eq:Qinf-def} equals the displayed prime, pole, and archimedean source
expression.

The Guinand--Weil explicit formula in the normalization of Connes, Bombieri,
and \ccm\ \cite{Guinand1948,Weil1952,Connes1999,Bombieri2000,
ConnesConsaniMoscovici2025} states, for every test function in the class of
Lemma~\ref{lem:admissible},
\[
\sum_{z\in Z_\zeta^*}g_v(z)
  =-\frac1\pi\sum_{q=p^a}\frac{\Lambda(q)}{\sqrt q}
        \wh g_v\!\left(\frac{\log q}{2\pi}\right)
     +2g_v(i/2)
      +\frac{1}{2\pi}\int_{\R}h_+(r)g_v(r)\dd r .
\]
Since $\supp\wh g_v\subset[-\Delta,\Delta]=[-L/(2\pi),L/(2\pi)]$, prime powers
$q>c$ satisfy $\log q/(2\pi)>\Delta$ and do not contribute.  This identifies
the finite form with the zero side.
\end{proof}

\begin{remark}[What is new here]\label{rem:prior}
Once $g_v$ is known to be admissible, the displayed identity is the explicit
formula applied to $g_v$; in that sense the zero-sum representation is the
classical mechanism of Weil \cite{Weil1952}, in the finite-rank setting
studied by Yoshida \cite{Yoshida1992} and built by \cvs/\ccm.  The content of
Theorem~\ref{thm:dictionary} is the exact closed-form transport
$v\mapsto g_v$ and its calculus: the truncated matrix is not merely
positive-semidefinite-approximating data, but evaluates exact zero sums for an
explicitly parametrized $(N{+}1)$-dimensional family of band-limited test
functions, with the source dependence factoring through the exact
$2N{+}1$-dimensional quotient of Corollary~\ref{cor:source-quotient}.  This is
the finite, inverse-free part of the correspondence between coefficient
vectors and test functions; no claim is made about realizing arbitrary
Guinand--Weil test functions this way.
\end{remark}

\begin{corollary}[Pole-neutral source survival]\label{cor:pole-neutral}
Let $\beta=L/(4\pi)$ and, for $s\ge0$, put
\[
        H_s(N)=\{v:M_0(v)=M_2(v)=\cdots=M_{2s}(v)=0\},
\]
where
\[
        M_0(v)=v_0+\sqrt2\sum_{k=1}^N v_k,\qquad
        M_{2j}(v)=\sqrt2\sum_{k=1}^N k^{2j}v_k\quad(j\ge1).
\]
Define the pole-neutral hyperplane
\[
        P_N(c)=
        \left\{
        v:\frac{v_0}{\beta^2}
          +\sqrt2\sum_{k=1}^N\frac{v_k}{k^2+\beta^2}=0
        \right\}.
\]
If $N\ge s+2$, then
\[
        \dim(H_s(N)\cap P_N(c))=N-s-1>0.
\]
Every nonzero vector in this space produces a nonzero test function $g_v$
with $g_v(i/2)=0$, and $g_v=g_w$ implies $v=w$ or $v=-w$.  On this family the
dictionary of Theorem~\ref{thm:dictionary} has no pole term.
\end{corollary}

\begin{proof}
Direct evaluation of the pole source \eqref{eq:pole-source} at integer
frequencies, and of its $x$-derivative at integer frequencies, gives
\[
        \psi_0(n)=
        C_c\frac{n}{n^2+\beta^2},
\qquad
        \psi_0'(n)=
        C_c\frac{\beta^2-n^2}{(n^2+\beta^2)^2},
\qquad
        C_c=\frac{L(\sqrt c+1/\sqrt c-2)}{2\pi^2}>0.
\]
Both evaluations are needed: off-diagonal entries of $Q_{\rm pole}$ use only
the integer values $\psi_0(n)$, but the diagonal entries are
$\psi_0'(n)$, and value agreement at integers alone would not determine them
(the diagonal involves all Fourier components of the source; see
\cite[Sec.~4.2]{ConnesvanSuijlekom2025}).  With both identities,
$Q_{\rm pole}$ is exactly $C_c$ times the divided-difference matrix of
$f(x)=x/(x^2+\beta^2)$, including the diagonal $f'(n)$.  From the partial
fractions $f(x)=\tfrac12\bigl((x-i\beta)^{-1}+(x+i\beta)^{-1}\bigr)$, the
divided difference is
\[
 \frac{f(m)-f(n)}{m-n}
 =\frac{\beta^2-mn}{(m^2+\beta^2)(n^2+\beta^2)},
\]
with the $m=n$ limit $f'(m)$; contracting against the even embedding, and
using $\sum_m u_m\,m/(m^2+\beta^2)=0$ by symmetry,
\[
\langle v,Q_{\rm pole}v\rangle
 =
 C_c\beta^2
 \left(
 \frac{v_0}{\beta^2}
 +\sqrt2\sum_{k=1}^N\frac{v_k}{k^2+\beta^2}
 \right)^2
 =
 2g_v(i/2).
\]
Thus $P_N(c)$ is exactly the pole-neutral row.

The row defining $P_N(c)$ is independent of the moment rows when $N\ge s+1$:
otherwise there would be a polynomial $P$ of degree at most $s$ with
$P(k^2)=1/(k^2+\beta^2)$ for $0\le k\le N$; then
$(x+\beta^2)P(x)-1$, of degree at most $s+1$, would vanish at the $N+1\ge s+2$
distinct points $k^2$, forcing it to vanish identically, which is impossible.
The dimension formula follows.

Finally, $K_v=2(T_v*T_v)$ in the Volterra convolution algebra of analytic
germs at $0$.  This algebra is an integral domain: if $\phi,\chi$ have lowest
nonvanishing orders $j,k$, the lowest coefficient of $\phi*\chi$ is a nonzero
beta-integral multiple of the product of lowest coefficients.  Hence
$K_v\equiv0$ implies $T_v=0$, so $v=0$.  If $g_v=g_w$, Fourier injectivity
gives $K_v=K_w$ on $[0,1]$, hence as germs; then
$(T_v-T_w)*(T_v+T_w)=0$, so $T_v=\pm T_w$ and $v=\pm w$.
\end{proof}

\subsection{A worked example over the first 512 zeros}\label{sec:example}

Take $c=13$, $N=4$, and the pole- and moment-neutral vector
$v\in H_0(4)\cap P_4(13)$ with $(v_2,v_3,v_4)=(1,0,-3)/\sqrt2$ and
$(v_0,v_1)$ determined by the two linear conditions; numerically
\[
 v=(-0.0859452,\;1.4749860,\;0.7071068,\;0,\;-2.1213203).
\]
The chain $v\mapsto T_v\mapsto K_v\mapsto\wh g_v\mapsto g_v$ for this vector
is displayed in Figure~\ref{fig:dictionary}.  The cutoff-free contraction,
computed from the closed-form \ccm\ assembly at 40 digits, is
\[
 \langle v,\Qinf v\rangle
 =0.049968414571096979730\ldots,
\]
and Theorem~\ref{thm:dictionary} asserts that this number is the sum of
$g_v$ over the nontrivial zeros; since every zero in the range used (and far
beyond $\gamma_{512}$) is known to lie on the critical line, the zero side is
$2\sum_{n\ge1}g_v(\gamma_n)$ over the ordinates.
Table~\ref{tab:zeroside} shows the partial sums against the first $512$
zeros: the raw residual falls to $-4.7\cdot10^{-11}$, and correcting the
remainder by the archimedean tail integral of Section~\ref{sec:tail} (the
smooth zero-density approximation to the missing zeros) leaves
$-5.4\cdot10^{-12}$.  The same package verifies the identity on a generic,
non-neutral vector at $(c,N)=(29,6)$, where the two sides of
Lemma~\ref{lem:entry-id} agree to $2.0\cdot10^{-10}$, and on its pole-neutral
counterpart, where the pole term vanishes to $9\cdot10^{-41}$ while the
prime and archimedean terms are of order one: the pole really is switched
off, not merely small.

\begin{table}[t]
\centering
\begin{tabular}{rccc}
\toprule
$M$ & $2\sum_{n\le M}g_v(\gamma_n)$ & raw residual & tail-corrected residual\\
\midrule
$32$  & $0.049967509\ldots$ & $-9.1\cdot10^{-7}$  & $-1.6\cdot10^{-7}$\\
$64$  & $0.049968326\ldots$ & $-8.8\cdot10^{-8}$  & $-9.4\cdot10^{-9}$\\
$128$ & $0.049968406\ldots$ & $-7.9\cdot10^{-9}$  & $-1.0\cdot10^{-9}$\\
$256$ & $0.049968413\ldots$ & $-6.5\cdot10^{-10}$ & $-8.3\cdot10^{-11}$\\
$512$ & $0.049968414\ldots$ & $-4.7\cdot10^{-11}$ & $-5.4\cdot10^{-12}$\\
\bottomrule
\end{tabular}
\caption{Zero-side partial sums for the worked example
($c=13$, $N=4$, pole-neutral $v$; $\gamma_n$ = ordinate of the $n$-th
nontrivial zero, $\gamma_{512}=826.90\ldots$).  The limit is the closed-form
matrix contraction $\langle v,\Qinf v\rangle$; residuals are signed.}
\label{tab:zeroside}
\end{table}

\begin{figure}[t]
\centering
\includegraphics[width=\textwidth]{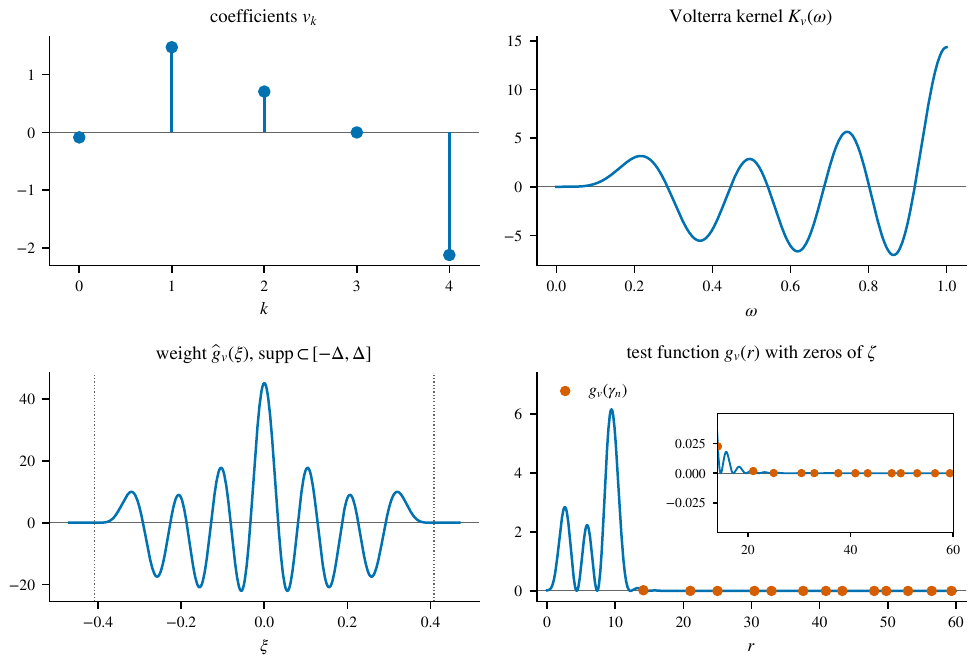}
\caption{The finite dictionary at $c=13$, $N=4$ for the worked example, read
in row order.  Top row: the coefficient vector $v$; the Volterra kernel $K_v$
on $[0,1]$.  Bottom row: the compactly supported Fourier weight $\wh g_v$ on
$[-\Delta,\Delta]$; the induced entire test function $g_v$ on $[0,60]$ with
the first ordinates $\gamma_n$ marked (inset: the small tail oscillation that
carries the zero sum).  By Theorem~\ref{thm:dictionary} the
quadratic value $\langle v,\Qinf v\rangle$ equals twice the sum of the marked
values (continued over all zeros).}
\label{fig:dictionary}
\end{figure}

\section{Archimedean tail order}\label{sec:tail}

Recall $\rho=2\pi/L$ and $I_N=\{-N,\ldots,N\}$.  For $T>\rho N$, set
$a_T=T/\rho$ and define
\[
        p_T(n)=\frac1{a_T-n},\qquad q_T(n)=\frac1{a_T+n}
        \quad(n\in I_N).
\]

\begin{lemma}[The archimedean density]\label{lem:hplus-positive}
The function $h_+$ is strictly increasing on $(0,\infty)$, with
\[
 h_+'(t)
 =
 \frac{t}{2}\sum_{k=0}^{\infty}
 \frac{k+1/4}{\bigl((k+1/4)^2+(t/2)^2\bigr)^2}
 \;\le\;
 \frac1t+\frac{13}{10\,t^2}
 \qquad(t>0).
\]
Interval evaluation in Arb \cite{Johansson2017Arb} gives
$h_+(7)=0.10717967\ldots>0$, so $h_+(t)>0$ for $t\ge7$; combined with the
derivative bound,
\[
        h_+(t)\;\le\;\log t-\tfrac{8}{5}
        \qquad(t\ge7).
\]
\end{lemma}

\begin{proof}
Differentiating the partial-fraction expansion of
$\operatorname{Re}\psi_\Gamma(1/4+it/2)$ term by term gives the displayed
series, which is positive for $t>0$; monotonicity and, with the certified
value $h_+(7)>0$, positivity on $[7,\infty)$ follow.  For the upper bound,
put $u=t/2$ and $g_u(x)=(x+1/4)/\bigl((x+1/4)^2+u^2\bigr)^2$.  The function
$s\mapsto s/(s^2+u^2)^2$ increases up to $s=u/\sqrt3$ and decreases
thereafter, so $g_u$ is unimodal in $x\ge0$ and
\[
 \sum_{k\ge0}g_u(k)\;\le\;\int_0^\infty g_u(x)\dd x+\max_{s>0}\frac{s}{(s^2+u^2)^2}
 \;=\;\frac{1}{2(u^2+1/16)}+\frac{9}{16\sqrt3\,u^3}
 \;\le\;\frac1{2u^2}+\frac{9}{16\sqrt3\,u^3},
\]
whence $h_+'(t)=u\sum_k g_u(k)\le 1/(2u)+9/(16\sqrt3\,u^2)
=1/t+ (9/(4\sqrt3))/t^2\le 1/t+13/(10t^2)$.
Integrating from $7$ to $t$,
\[
 h_+(t)\le h_+(7)+\log\frac t7+\frac{13}{10}\Bigl(\frac17-\frac1t\Bigr)
 \le \log t+\Bigl(h_+(7)+\tfrac{13}{70}-\log7\Bigr)
 \le \log t-1.65,
\]
using the certified $h_+(7)<0.1072$.
\end{proof}

\begin{theorem}[Archimedean tail order]\label{thm:tail-order}
Let $T_2>T_1>\max(\rho N,7)$.  Then
\[
        \Delta_{T_1,T_2}:=Q_{{\rm arch},T_2}-Q_{{\rm arch},T_1}
\]
has the exact representation
\[
\Delta_{T_1,T_2}
 =
 \frac1{\pi^2}\int_{T_1}^{T_2}
 h_+(T)\frac{\sin^2(LT/2)}{\rho}
        \bigl(p_Tp_T^{\mathsf t}+q_Tq_T^{\mathsf t}\bigr)\dd T .
\]
Consequently $\Delta_{T_1,T_2}$ is positive definite on the full complex
finite frequency space indexed by $I_N$, and hence on the real even subspace.
Moreover, in the natural order $-N<\cdots<N$, every minor of
$\Delta_{T_1,T_2}$ with increasing row and column sets is strictly positive.
\end{theorem}

\begin{proof}
Since the finite-$T$ archimedean integrand is even in $r$,
\[
        \frac{d}{dT}\psi_{{\rm arch},T}(x)
        =\frac{1}{\pi^2}h_+(T)\,\mathcal S(T,x,L).
\]
At an integer node $n\in I_N$, direct integration gives
\[
 \mathcal S(T,n,L)
 =
 \frac{2\rho n\sin^2(LT/2)}{T^2-\rho^2n^2}
 =
 \frac{2\sin^2(LT/2)}{\rho}\,\frac{n}{a_T^2-n^2},
\]
and for the diagonal entries, differentiating under the integral sign and
integrating by parts,
\[
 \partial_x\mathcal S(T,x,L)\Big|_{x=n}
 =\int_0^L 2\pi\Bigl(1-\frac yL\Bigr)
   \cos\Bigl(2\pi n\Bigl(1-\frac yL\Bigr)\Bigr)\cos(Ty)\dd y
 =\frac{2\sin^2(LT/2)}{\rho}\,\frac{a_T^2+n^2}{(a_T^2-n^2)^2},
\]
which is exactly the derivative at $x=n$ of the closed form
$x\mapsto(2\sin^2(LT/2)/\rho)\,x/(a_T^2-x^2)$.  This matters because the
diagonal entries of a divided-difference matrix are derivatives of the true
source, not of an integer-node surrogate; here the two agree.  The divided
difference of $x/(a^2-x^2)$ is
\[
 \frac12\left(
 \frac1{(a-m)(a-n)}+\frac1{(a+m)(a+n)}
 \right),
\]
with the same formula in the diagonal limit, so
\[
 \frac{d}{dT}(Q_{{\rm arch},T})_{mn}
 =
 h_+(T)\frac{\sin^2(LT/2)}{\pi^2\rho}
 \left(
 \frac1{(a_T-m)(a_T-n)}
 +\frac1{(a_T+m)(a_T+n)}
 \right),
\]
which is the displayed rank-two density.

For positive definiteness, suppose $x\in\C^{I_N}$ and
$x^*\Delta_{T_1,T_2}x=0$.  The integrand is a sum of two absolute squares
with the strictly positive scalar factor $h_+(T)\sin^2(LT/2)$ off a discrete
set, so by continuity
\[
        \sum_{n\in I_N}\frac{x_n}{a-n}=0
\]
for all $a=T/\rho$ in $(T_1/\rho,T_2/\rho)$.  This rational function has
simple poles at the nodes $n\in I_N$ unless it vanishes identically; taking
residues gives every $x_n=0$.

It remains to prove the minor statement.  Put $A=T_1/\rho$ and $B=T_2/\rho$.
After the change of variables $T=\rho a$, the same representation is
\[
 \Delta_{T_1,T_2}(m,n)
 =
 \int_{[-B,-A]\cup[A,B]}\frac{1}{(s-m)(s-n)}\dd\nu(s),
\]
where $\nu$ is the push-forward of
$\pi^{-2}h_+(\rho a)\sin^2(\pi a)\dd a$
to both $s=a$ and $s=-a$: a nonnegative measure with positive mass on every
nonempty open subinterval.  For increasing row nodes $x_1<\cdots<x_k$ and
column nodes $y_1<\cdots<y_k$ in $I_N$, Andr\'eief's identity
\cite{Andreief1886,Forrester2019} gives
\[
 \det[\Delta_{T_1,T_2}(x_i,y_j)]
 =
 \frac1{k!}\int
 \det\!\left[\frac1{s_\ell-x_i}\right]_{i,\ell=1}^k
 \det\!\left[\frac1{s_\ell-y_j}\right]_{j,\ell=1}^k
 \prod_{\ell=1}^k\dd\nu(s_\ell).
\]
On ordered distinct support points $s_1<\cdots<s_k$, the Cauchy determinant
formula \cite[Ch.~1]{Karlin1968} gives
\[
 \det\!\left[\frac1{s_\ell-x_i}\right]_{i,\ell=1}^k
 =
 \frac{\prod_{i<j}(x_i-x_j)\prod_{\ell<r}(s_r-s_\ell)}
      {\prod_{i,\ell}(s_\ell-x_i)} .
\]
All nodes $x_i$ lie strictly between the two support intervals, so if
$\ell_-$ of the $s_\ell$ lie in $[-B,-A]$ the sign of this determinant is
$(-1)^{k(k-1)/2+k\ell_-}$, depending only on $k$ and $\ell_-$.  The same
holds for the $y_j$ determinant with the same $s_\ell$, hence the same
$\ell_-$, so the two signs cancel and the integrand is nonnegative, and
positive on a set of positive $\nu^{\otimes k}$ measure.  The integral is
therefore strictly positive.  This is the classical Cauchy total-positivity
mechanism \cite{Karlin1968,Simon2014CauchyTP,BertolaGekhtmanSzmigielski2009},
applied to the finite archimedean tail representation above.
\end{proof}

\begin{corollary}[A two-sided finite certification rule]\label{cor:tail-cert}
On the full finite frequency space $\C^{I_N}$, let
\[
        Q_T^{\rm tot}=Q_{\rm prime}^{(c)}+Q_{\rm pole}+Q_{{\rm arch},T},
\qquad
 B_T=
 \frac1{\pi^2}\int_T^\infty
 h_+(r)\frac{\sin^2(Lr/2)}{\rho}
        \bigl(\|p_r\|_2^2+\|q_r\|_2^2\bigr)\dd r .
\]
For $T>\max(\rho N,7)$:
\begin{enumerate}
\item[(i)] $\Qinf-Q_T^{\rm tot}\succ0$ and
$0\preceq \Qinf-Q_T^{\rm tot}\preceq B_TI$, hence
\[
        \lambda_j(Q_T^{\rm tot})<\lambda_j(\Qinf)
        \le\lambda_j(Q_T^{\rm tot})+B_T
        \qquad(1\le j\le 2N+1),
\]
and the eigenvalues $\lambda_j(Q_T^{\rm tot})$ increase strictly to $\lambda_j(\Qinf)$ as $T\uparrow\infty$.
\item[(ii)] \emph{Decision rule.}
$\lambda_j(Q_T^{\rm tot})\ge0$ certifies $\lambda_j(\Qinf)>0$;
$\lambda_j(Q_T^{\rm tot})<-B_T$ certifies $\lambda_j(\Qinf)<0$;
a negative eigenvalue in the band $[-B_T,0)$ certifies nothing about the
cutoff-free sign.
\item[(iii)] \emph{Explicit budget.}  For $N\ge1$,
\[
 B_T\;\le\;\frac{2(2N{+}1)\rho}{\pi^2}
 \left[
 \frac{\log T}{T-\rho N}
 +\frac{1}{\rho N}\log\frac{T}{T-\rho N}
 \right],
\]
and, as $T\to\infty$ with $(c,N)$ fixed,
\[
 B_T=\frac{(2N{+}1)\rho}{\pi^2\,T}\Bigl(\log\frac{T}{2\pi}+1\Bigr)(1+o(1)).
\]
In particular $B_{2T}/B_T\to\tfrac12$: each doubling of the archimedean
cutoff lowers the certification floor by one binary digit.
\end{enumerate}
The corresponding restricted statements hold on the real even subspace.
\end{corollary}

\begin{proof}
(i) Integrate Theorem~\ref{thm:tail-order} from $T$ to infinity.  The
improper integral converges entrywise because
$p_r(n),q_r(n)=O(r^{-1})$ uniformly for $n\in I_N$ while $h_+(r)=O(\log r)$;
strict positivity follows by restricting to any compact window, and the trace
of the positive tail is $B_T$, so the tail is bounded above by $B_TI$.  The
eigenvalue inequalities and monotonicity follow from Weyl monotonicity.
(ii) is immediate from (i).
(iii) For the explicit bound, use $\sin^2\le1$, the envelope
$h_+(r)\le\log r$ of Lemma~\ref{lem:hplus-positive}, and
$\|p_r\|_2^2+\|q_r\|_2^2\le 2(2N{+}1)\rho^2/(r-\rho N)^2$, then integrate
\[
 \int_T^\infty\frac{\log r}{(r-b)^2}\dd r
 =\frac{\log T}{T-b}+\frac1b\log\frac{T}{T-b}
 \qquad(b=\rho N<T).
\]
For the asymptotic form, $\|p_r\|_2^2+\|q_r\|_2^2
=2(2N{+}1)\rho^2r^{-2}(1+O(N\rho/r))$,
$h_+(r)=\log(r/2\pi)+o(1)$ \cite{DLMFDigammaAsymptotic}, and the oscillatory
part of $\sin^2=\tfrac12-\tfrac12\cos(Lr)$ contributes $O(T^{-2}\log T)$
after one integration by parts; the remaining smooth integral is
$\tfrac12\int_T^\infty\log(r/2\pi)\,2(2N{+}1)\rho\,r^{-2}\dd r
=(2N{+}1)\rho\,(\log(T/2\pi)+1)/T$ up to the stated relative error.
\end{proof}

\paragraph{The $T=800$ correction scale.}
For $c=100$, $N=200$, $T=800$, the released package gives
$\rho N=272.875270768\ldots<800$, a certified dyadic interval budget
$B_{800}<0.897$ (the interval certificate bounds $\sin^2(Lr/2)\le1$ and
$h_+(r)\le\log r$, roughly doubling the exact value $0.4203$, against which
the asymptotic form of (iii) predicts $0.4051$; the closed-form bound of
(iii) gives the weaker $1.58$), and entrywise tail bound $3.18\cdot10^{-3}$.
A separate cutoff-free Arb interval $LDL^{\mathsf T}$ certificate at the same
$(c,N)$, at $9000$ bits, gives $n_+=401$ and $n_-=0$.  Thus the finite-$T$
negative eigenvalues that motivated this paper sat deep inside the
inconclusive band of the decision rule and were tail artifacts, while the
cutoff-free matrix is certified positive.  The asymptotic of (iii) makes the
obstruction quantitative: driving $B_T$ below a spectral scale of $10^{-59}$
at $(c,N)=(100,200)$ would require $T\approx8\cdot10^{62}$.  Deep spectral
scales are not reachable through the archimedean cutoff; they are
reachable through the cutoff-free closed forms.

\begin{figure}[t]
\centering
\includegraphics[width=\textwidth]{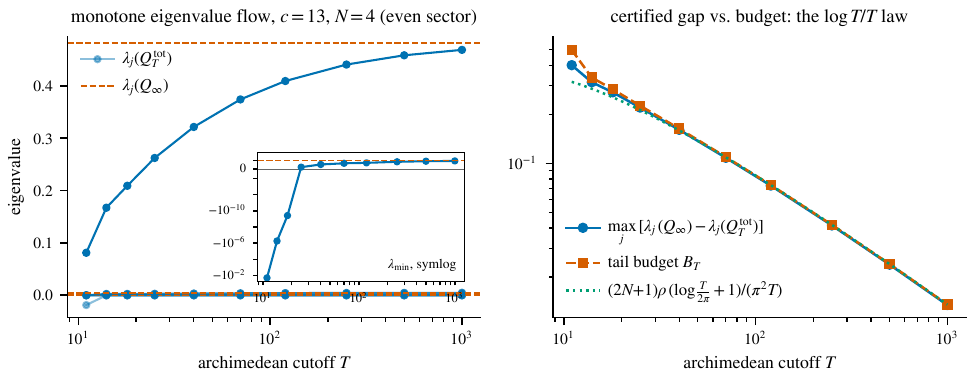}
\caption{The tail order in action at $c=13$, $N=4$.
Left: the eigenvalues of $Q_T^{\rm tot}$ increase strictly
to those of $\Qinf$ (dashed lines) as the cutoff $T$ grows; by
Corollary~\ref{cor:tail-cert} every gap is at most $B_T$.  Inset
($\lambda_{\min}$, symmetric log scale): at small $T$ the finite-cutoff
matrix has genuinely negative eigenvalues ($-1.9\cdot10^{-2}$ at $T=11$,
$-5.3\cdot10^{-7}$ at $T=14$, $-3.9\cdot10^{-10}$ at $T=18$), each far
inside its inconclusive band $(-B_T,0)$; the cutoff-free limit is
$+9.7\cdot10^{-15}$.  Right: the largest gap
$\max_j[\lambda_j(\Qinf)-\lambda_j(Q_T^{\rm tot})]$ against the budget $B_T$
and its asymptotic form $(2N{+}1)\rho(\log(T/2\pi)+1)/(\pi^2T)$; the
$\log T/T$ decay is the one-binary-digit-per-doubling law.}
\label{fig:tailorder}
\end{figure}

\section{Verification and scope}\label{sec:verification}

The released package records exact integer or symbolic audits for the
single-frequency identity, the $2N{+}1$ source quotient through $N=30$, and the
pole-neutral support corollary, including the diagonal derivative identity of
Corollary~\ref{cor:pole-neutral}.  The pole-neutral audit checks $3311$
pole-square entries, $18$ moment-independence determinant cases, and $45$
dimension instances of the formula $\dim=N-s-1$.  The archimedean audits
include an Arb interval certification of $h_+(7)>0$, a three-route derivative
bridge for the rank-two density (symbolic residual $0$; direct-kernel error
$9.5\cdot10^{-125}$), strict-total-positivity smoke tests with zero bad
minors for $N=2,3$ ($251$ and $3431$ minors), the $c=100,N=200,T=800$
interval tail budget quoted above, and the cutoff-free Arb
$LDL^{\mathsf T}$ certificate with $n_-=0$ at $9000$ bits, whose generating
script ships with the package.  The dictionary itself is confirmed by three
independent routes: the closed-form \ccm\ assembly, the source-side
evaluation of Theorem~\ref{thm:dictionary}, and the zero-side sum of
Table~\ref{tab:zeroside}; on a generic vector at $(c,N)=(29,6)$ the first two
routes agree to $2.0\cdot10^{-10}$, and all displayed constants are generated
by scripts included in the package.  These checks guard constants and signs;
they are not proof substitutes.

The limitations are explicit.  The dictionary is one-way, from finite
Galerkin vectors into the Guinand--Weil class; no inverse map from arbitrary
test functions is claimed.  Strict observability and strict total positivity
concern the isolated post-band archimedean increment, not the full Weil
block or the prime block.  The paper does not prove
RH, Weil positivity, a prime-location bound, a next-prime theorem, or a
factoring result.  The finite source quotient is the coordinate handoff for a
follow-up finite event-calculus study; no event-recovery theorem is used
here.

\paragraph{Use of AI tools.}
Language-model tools were used as a drafting and verification-tooling
assistant under the author's direction; the author takes sole responsibility
for all mathematics, computations, and text.

\paragraph{Data and code availability.}
The verification package (scripts, JSON artifacts, checksums, and the figure
generators) accompanies this paper as ancillary files and is maintained at
\url{https://github.com/akivag613/connes-cvs-} under
\texttt{papers/2\_guinand\_weil\_dictionary\_tail\_order/}; a checksummed archive is
deposited on
Zenodo (concept DOI \href{https://doi.org/10.5281/zenodo.21124802}{10.5281/zenodo.21124802}).

{\small\raggedright
\bibliographystyle{plainurl}
\bibliography{main}
}

\end{document}